\pdfoutput=1
\documentclass[a4paper]{article}
\usepackage[T1]{fontenc}
\usepackage[utf8x]{inputenc}
\usepackage{amsmath,amssymb,amsfonts,amsthm,mathtools,array,booktabs,caption,microtype,graphicx,xcolor}
\usepackage{cite}
\usepackage{tikz}
\usetikzlibrary{calc,trees,arrows,shapes.misc}
\tikzset{level 1/.style={level distance=1.5cm, sibling distance=3.5cm}}
\tikzset{level 2/.style={level distance=1.5cm, sibling distance=2cm}}
\usetikzlibrary{decorations.markings}
\usetikzlibrary{decorations.pathreplacing}
\usepackage[pdfusetitle,hidelinks]{hyperref}
\usepackage{bookmark}
\hypersetup{linktoc = all}
\hypersetup{bookmarksnumbered}
\usepackage{enumerate}

\newtheorem{theorem}{Theorem}[section]

\newcommand{\definetheorem}[2]{%
  \newtheorem{#1}[theorem]{#2}%
  \expandafter\newcommand\csname #1autorefname\endcsname{#2}}
\definetheorem{proposition}{Proposition}
\definetheorem{lemma}{Lemma}
\definetheorem{definition}{Definition}
\definetheorem{remark}{Remark}
\definetheorem{corollary}{Corollary}
\newcommand\bee{\begin{equation}}
\newcommand\eeq{\end{equation}}

\newcommand{\CC}{{\mathbb{C}}}

\newcommand{\myoverline}[1]{\mathrlap{\overline{\phantom{#1}}}#1}
\newcommand{\mubar}{\myoverline{\mu}}
\newcommand{\nubar}{\myoverline{\nu}}

\DeclareMathOperator{\Tr}{Tr}

\DeclareMathOperator{\rank}{rank}
\DeclareMathOperator{\adj}{adj}
\DeclareMathOperator{\roots}{roots}

\numberwithin{equation}{section}

\makeatletter
\newcommand{\yngantisym}{{\let\@nomath\@gobble\let\typeout\@gobble\tiny\yng(1,1)}}
\makeatother

\newcommand{\lie}{\mathfrak}

\newcommand{\FrobSch}{\operatorname{FS}}
\newcommand{\orthrootThm}{\beta}
\newcommand{\refrootLem}{\alpha}
\newcommand{\orthrootLem}{\beta}
\newcommand{\farrootLem}{\gamma}
\newcommand{\orthrootSec}{\beta}
\newcommand{\farrootSec}{\gamma}
\newcommand{\anyroot}{\alpha}
\usepackage[all,cmtip]{xy}
\newcommand{\fundef}[5]{
\entrymodifiers={+!!<0pt,\fontdimen22\textfont2>}
\xymatrix@R=3pt{\llap{$#1$\;\;} {#2} \ar@{->}[r] & {#3} \\ {#4} \ar@{|->}[r] & {#5}}
} 

\newcommand{\eps}{\varepsilon}

\title{Adjoints in symmetric squares of\\Lie algebra representations}
\author{Bruno Le Floch and Ilia Smilga\thanks{The second author was supported by the Simons Investigator Award 409735 from the Simons Foundation.}}
\date{\today}
\hypersetup{pdfauthor = {Bruno Le Floch and Ilia Smilga}}

\begin{document}

\maketitle

\begin{abstract}
  Given finite-dimensional complex representations $V$ and $V'$ of a simply-connected semisimple compact Lie group $G$, we determine the dimension of the $G$-invariant subspace of $\mathrm{adj}(G)\otimes V\otimes V'$, of $\mathrm{adj}(G)\otimes S^2 V$, and of $\mathrm{adj}(G)\otimes\Lambda^2 V$, where $\mathrm{adj}(G)$ is the adjoint representation.
  In other words we derive the multiplicity with which summands of $\mathrm{adj}(G)$ appear in a tensor product $V \otimes V'$ or (anti)symmetric square $S^2 V$ or $\Lambda^2 V$.
  We find in particular that the dimension of the $G$-invariant subspace of $\mathrm{adj}(G)\otimes S^2 V$ is larger than (resp.\ smaller or equal to) that of $\mathrm{adj}(G)\otimes\Lambda^2 V$ for a symplectic (resp.\ orthogonal) representation $V$.
\end{abstract}

\section{Introduction}

\subsection{Disclaimer}

After submitting this paper, we learned that most of its results are in fact already known \cite{BZ88, PY08}. We thank the anonymous referee for pointing this out to us. We decided to keep this paper online, adding this disclaimer, for continuity reasons (and because an arXiv paper cannot be completely deleted, anyway).  The only possibly novel result of this paper is the detailed classification in the case $\mu \neq \overline{\nu}$ of Theorem~\ref{thm:adjoints_in_tensor_square}, which does not warrant publication on its own. Had we found the relevant references earlier, this paper would have never been written.

\subsection{Motivation}

In gauge theories, only gauge-invariant operators are physically meaningful.
These arise mathematically as $G$-invariants inside tensor products of representations of the gauge group, a compact simple Lie group~$G$.
This paper arose from a study of exactly marginal deformations of a class of three-dimensional superconformal field theories in~\cite{LLF}.
There, the dimensions of $G$-invariant subspaces of $\adj(G)\otimes S^2 V$ and $\adj(G)\otimes\Lambda^2 V$ were compared for irreducible representations~$V$ of~$G$, at a physicist's level of rigour.
The present article provides a rigorous case by case analysis proving this inequality, which is stated as the equivalence (\ref{itm:V_mu_sympl}) $\iff$ (\ref{itm:dim_order}) in Corollary~\ref{reformulation_as_integrals}.

In the physical context of interest in \cite{LLF}, the two sides of the inequality (\ref{itm:dim_order}) count gauge-invariant operators obtained as products of one gauge multiplet (in $\adj(\mathfrak{g})$) and two matter multiplets (in $V$), with bosonic and fermionic statistics on the left and right-hand side, respectively. A bosonic and a fermionic gauge-invariant operators can assemble into so-called ``long'' representations of superconformal symmetry, and any remaining operator must be in a special ``short'' representation, whose number was sought.  The inequality understood in \cite{LLF} gave a lower bound on that number.
In this work, we provide explicit descriptions and dimensions of these $G$-invariant subspaces, and also treat the case of reducible representations by analysing $G$-invariant subspaces of more general tensor products $\adj(G)\otimes V\otimes V'$ for pairs of irreducible representations $V,V'$.

Along the way, we reproduce in our \autoref{thm:adjoints_in_tensor_square} the dimension of the $G$-invariant subspace of $\adj(G)\otimes V\otimes V$ already determined in \cite{KW}.  There, this dimension was resolved into contributions from $\adj(G)\otimes S^2 V$ and $\adj(G)\otimes\Lambda^2 V$ for all compact simple Lie groups except $E_6$, $E_8$, $F_4$.
The proof relied on diverse approaches:
Schur function techniques for $G=SU(k+1)$,
realizing $\adj(G)\otimes V\otimes V$ as a sub-representation of $(\Box\otimes V)^{\otimes 2}$ with $\Box$ being the defining representation for $G=SO(2k+1),USp(2k),SO(4k)$, or another suitable representation for $G=G_2,E_7$,
and extending $G=SO(4k+2)$ to $O(4k+2)$.
In the remaining cases of $E_6$, $E_8$, $F_4$, the authors of~\cite{KW} proposed conjectures based on limited data concerning specific representations.
Our \autoref{thm:adj_in_sym_or_antisym} provides a uniform result for all Lie groups, which we obtain straightforwardly by embedding an irreducible representation~$V$ as a highest-weight summand in a tensor product of representations whose dominant weights are multiples of some fundamental weights.
This allows us to establish the conjecture of~\cite{KW} on $E_8$ and~$F_4$ and disprove their conjecture on~$E_6$.

\subsection{Basic definitions}

In all of this paper, $\lie{g}$ is some semisimple compact Lie algebra, and $G$ is the simply-connected group with algebra $\lie{g}$. We choose in $\lie{g}$ a Cartan subalgebra $\lie{h}$, that we equip with the Killing form (and thus identify to $\lie{h}^*$).

We choose one of the closed Weyl chambers in $\lie{h}$ as the dominant one: we denote it by $\lie{h}^+$, and its interior by $\lie{h}^{++}$. We call $\alpha_1, \ldots, \alpha_r$ the corresponding simple positive roots (whose kernels are the walls of $\lie{h}^+$); we will use the Bourbaki numbering in the introduction and in Section~\ref{sec:square}, but switch to a special ad hoc numbering in Section~\ref{sec:dual_and_sym}. We call $\rho$ the Weyl vector of $\lie{g}$ (half-sum of the positive roots, or sum of the fundamental weights).

All representations of $\lie{g}$ under consideration will be finite-dimensional and complex. We call a weight $\mu$ (i.e.\ an element of $\lie{h}^*$) \emph{dominant} (resp.\ \emph{integral}) if all of the scalar products $\langle \mu, \alpha_i^\vee \rangle$ (where $\alpha^\vee := 2\alpha/\langle \alpha, \alpha \rangle$) are non-negative (resp.\ integer). For each dominant integral weight $\mu$, we denote by $V_\mu$ the irreducible representation of $\lie{g}$ with highest weight $\mu$. We also denote by $\adj(\lie{g})$ the adjoint representation.

Since $\lie{g}$ is compact, for each representation $V$ of $\lie{g}$, its complex-conjugate $\overline{V}$ and its dual $V^*$ are isomorphic (because $V$ has an invariant Hermitian form). We consequently denote by $\mu \mapsto \overline{\mu}$ the opposition involution, so that $V_{\overline{\mu}} \simeq \overline{V_\mu} \simeq V^*_\mu$ for all $\mu$.

Finally, we recall the following definition:
\begin{definition}
Two roots of $\mathfrak{g}$ are said to be \emph{weakly orthogonal} if they are orthogonal and their sum is also a root of $\mathfrak{g}$.
\end{definition}

\begin{remark}
We defined $\lie{g}$ as a compact Lie algebra (because of the physical motivation of this paper), but of course all of our results apply equally well to complex Lie algebras. The only place where we really use the compact structure is the reformulation in terms of integrals (items (\ref{itm:FrobSch_integral}) and (\ref{itm:char_integral}) in Corollary~\ref{reformulation_as_integrals}).
\end{remark}

\subsection{Main result}

In this paper, we prove two theorems.
The case $\nubar=\mu$ of the first theorem already appears as \cite[Prop.~1.4]{BZ88} and as \cite[Prop.~4.2]{KW}.
Our second theorem already appears as \cite[Thm.~3.3]{PY08}; it also contains Propositions 5.4, 6.3--6.5, 7.1--7.2 and Conjecture 7.3 from \cite{KW}, as well as a modification of their Conjecture 7.4.

\begin{theorem}[Adjoints in a tensor product] \label{thm:adjoints_in_tensor_square}
  Given a semisimple Lie algebra~$\mathfrak{g}$ and two dominant weights $\mu$ and $\nu$, the $\mathfrak{g}$-invariant subspace of $\adj(\mathfrak{g})\otimes V_\mu\otimes V_\nu$ has dimension
  \begin{equation} \label{eq:multiplicity_in_tensor_square}
    \dim\Bigl(\bigl(\adj(\mathfrak{g})\otimes V_\mu\otimes V_\nu\bigr)^{\mathfrak{g}}\Bigr)
    = \begin{cases}
      \#\bigl\{j\bigm|\langle\mu,\alpha_j^\vee\rangle > 0 \bigr\} & \text{if } \nubar=\mu , \\
      0 \text{ or } 1 \text{ (see below)} & \text{if } \nubar-\mu\in\roots(\mathfrak{g}) , \\
      0 & \text{otherwise}.
    \end{cases}
  \end{equation}
More precisely, if $\nubar-\mu$ is equal to some root $\orthrootThm\in\roots(\mathfrak{g})$, the dimension is generally equal to~$1$. It is equal to~$0$ only in the following cases:
\begin{enumerate}
\item \label{itm:weakly_orth_overall} $\mathfrak{g}$ has a factor of type $B_n$, $C_n$ or~$F_4$, $\langle \mu, \alpha_j^\vee \rangle = 0$ for some short simple root $\alpha_j$ of that factor, and $\orthrootThm$ is weakly orthogonal to $\alpha_j$. (See Table~\ref{tab:weak_orth_enum} for the complete list of possible values of $\alpha_j$ and $\orthrootThm$).
\item \label{itm:g2_1} $\mathfrak{g}$ has a factor of type $G_2$, $\langle \mu, \alpha_1^\vee \rangle = 0$, and $\orthrootThm = -(\alpha_1 + \alpha_2)$ or $\orthrootThm = 2 \alpha_1 + \alpha_2$ (where $\alpha_1$ and $\alpha_2$ denote the simple roots of that factor, with $\alpha_1$ being short).
\item \label{itm:g2_2} $\mathfrak{g}$ has a factor of type $G_2$, $\langle \mu, \alpha_1^\vee \rangle = 1$, and $\orthrootThm = \alpha_1 + \alpha_2$ or $\orthrootThm = -(2 \alpha_1 + \alpha_2)$ (with $\alpha_1$ and $\alpha_2$ as above).
\end{enumerate}
\end{theorem}

\begin{table}
\caption{\label{tab:weak_orth_enum} Enumeration of roots $\orthrootThm$ that are weakly orthogonal to any short simple root $\alpha_j$ for simple $\lie{g}$ of type $B_n$, $C_n$ or~$F_4$. We follow Bourbaki \cite{BouGAL456} for the numbering and coordinate expressions of roots.}
\centering
\setlength{\tabcolsep}{.9\tabcolsep}
\begin{tabular}{ll@{\hspace{-20pt}}rl}
\toprule
$\mathfrak{g}$ & $j$ & $\alpha_j$ & $\orthrootThm$ \\
\midrule
$B_n$ & $n$ & $\eps_n$ & $\pm \eps_i$, $i \in \{1, \ldots, n-1\}$ \\
$C_n$ & $1, \ldots, n-1$ & $\eps_j - \eps_{j+1}$ & $\pm (\eps_j + \eps_{j+1})$ \\
$F_4$ & $3$ & $\eps_4$ & $\pm \eps_i$, $i \in \{1, 2, 3\}$ \\
$F_4$ & $4$ & $\frac{1}{2}(\eps_1 - \eps_2 - \eps_3 - \eps_4)$ & $\pm \left(\frac{1}{2}(\eps_1 + \eps_2 + \eps_3 + \eps_4) - \eps_i\right)$, $i \in \{2, 3, 4\}$ \\
\bottomrule
\end{tabular}
\end{table}

\begin{theorem}[Adjoints in a symmetric or antisymmetric square: Thm.~3.3 in \cite{PY08}]
\label{thm:adj_in_sym_or_antisym}
  Given a semisimple Lie algebra~$\mathfrak{g}$ and a dominant weight $\mu$, the $\mathfrak{g}$-invariant subspaces of $\adj(\mathfrak{g})\otimes S^2 V_\mu$ and $\adj(\mathfrak{g})\otimes \Lambda^2 V_\mu$ have the following dimensions.
  If $\mu=\mubar$,
  \begin{equation}
  \label{eq:adj_in_sym_or_antisym}
    \begin{aligned}
      \dim\Bigl(\bigl(\adj(\mathfrak{g})\otimes S^2V_\mu\bigr)^{\mathfrak{g}}\Bigr)
      & = \sum_{j\mid\langle\mu,\alpha_j^\vee\rangle > 0} \Bigl( \frac{1}{2}\delta_{\varpi_j\neq\overline{\varpi_j}} + \delta_{\varpi_j=\overline{\varpi_j}} \delta_{V_\mu\text{ symplectic}} \Bigr) ,
      \\
      \dim\Bigl(\bigl(\adj(\mathfrak{g})\otimes\Lambda^2V_\mu\bigr)^{\mathfrak{g}}\Bigr)
      & = \sum_{j\mid\langle\mu,\alpha_j^\vee\rangle > 0} \Bigl( \frac{1}{2}\delta_{\varpi_j\neq\overline{\varpi_j}} + \delta_{\varpi_j=\overline{\varpi_j}} \delta_{V_\mu\text{ orthogonal}} \Bigr) .
    \end{aligned}
  \end{equation}
  Otherwise these dimensions are zero.
\end{theorem}

As announced above, we use the latter theorem to derive the following corollary, whose first half (the equivalence (\ref{itm:V_mu_sympl}) $\iff$ (\ref{itm:dim_order})) is used in \cite{LLF}. The second half (the equivalence (\ref{itm:FrobSch_integral}) $\iff$ (\ref{itm:char_integral})) is simply a reformulation that we find elegant.

\begin{corollary}
\label{reformulation_as_integrals}
  Under the assumptions of \autoref{thm:adj_in_sym_or_antisym}, the following four properties are equivalent:
  \begin{enumerate}[(i)]
    \item \label{itm:V_mu_sympl} $V_\mu$ is symplectic;
    \item \label{itm:dim_order} $\displaystyle \dim\Bigl(\bigl(\adj(\mathfrak{g})\otimes S^2V_\mu\bigr)^{\mathfrak{g}}\Bigr)
      > \dim\Bigl(\bigl(\adj(\mathfrak{g})\otimes\Lambda^2V_\mu\bigr)^{\mathfrak{g}}\Bigr)$;
    \item \label{itm:FrobSch_integral} $\displaystyle \int_G \chi_{V_\mu}(g^2) d\eta(g) < 0$;
    \item \label{itm:char_integral} $\displaystyle \int_G \chi_{\adj(\mathfrak{g})}(g) \chi_{V_\mu}(g^2) d\eta(g) > 0$.
  \end{enumerate}
Here $\chi_V$ denotes the character of a representation $V$, and $\eta$ is the Haar measure on $G$.
\end{corollary}

\begin{proof}
The equivalence (\ref{itm:V_mu_sympl}) $\iff$ (\ref{itm:FrobSch_integral}) is classical \cite[\S II.6]{BtD}: in fact, the left-hand side of the inequality in (\ref{itm:FrobSch_integral}) is precisely the Frobenius-Schur indicator (see \eqref{eq:Frob_Schur_def}).

The implication (\ref{itm:dim_order}) $\implies$ (\ref{itm:V_mu_sympl}) immediately follows from the theorem. For (\ref{itm:V_mu_sympl}) $\implies$ (\ref{itm:dim_order}), it suffices to show that, for symplectic $V_\mu$, the $\delta_{\varpi_j=\overline{\varpi_j}}$ term in~\eqref{eq:adj_in_sym_or_antisym} is non-vanishing.  If it vanished, then $\langle\mu,\alpha_j^\vee\rangle=0$ for all self-dual weights~$\varpi_j$, so that $\mu$ would consist of non-self-dual weights only, hence would be a sum of terms of the form $\varpi_i+\overline{\varpi_i}$.  Since all such terms have Frobenius-Schur indicator equal to $+1$, the representation $V_\mu$ would then be orthogonal instead of symplectic.

To show the equivalence (\ref{itm:dim_order}) $\iff$ (\ref{itm:char_integral}), we start from the explicit expressions of characters of $S^2V_\mu$ and $\Lambda^2V_\mu$,
\begin{equation}
  \chi_{S^2V_\mu}(g) = \frac{1}{2} \Bigl( \chi_{V_\mu}(g)^2 + \chi_{V_\mu}(g^2) \Bigr) ,
  \qquad
  \chi_{\Lambda^2V_\mu}(g) = \frac{1}{2} \Bigl( \chi_{V_\mu}(g)^2 - \chi_{V_\mu}(g^2) \Bigr) .
\end{equation}
Their difference is $\chi_{V_\mu}(g^2)$.  By multiplicativity of characters, the integrand in~(\ref{itm:char_integral}) is the difference $\chi_{\adj(\mathfrak{g})\otimes S^2V_\mu}(g) - \chi_{\adj(\mathfrak{g})\otimes\Lambda^2V_\mu}(g)$.  Integrating a character against the Haar measure yields the dimension of the $\mathfrak{g}$-invariant subspace, which concludes the proof.
\end{proof}

The remaining two sections are devoted to proving the two theorems, respectively.
We also specialize \autoref{thm:adj_in_sym_or_antisym} to each classical and exceptional group in \autoref{ssec:KW} to reproduce the results and test the conjectures of~\cite{KW}.

\section{Proof of Theorem~\ref{thm:adjoints_in_tensor_square}: counting in the tensor product}
\label{sec:square}

\subsection{The Racah--Speiser algorithm}

We seek the multiplicity of the trivial representation in $\adj(\mathfrak{g})\otimes V_\mu \otimes V_\nu$. We will instead compute the multiplicity of $V_{\nubar}$ in $\adj(\mathfrak{g})\otimes V_\mu$, which is equal (decompose the latter representation into irreducibles, and apply Schur's lemma).

We start by recalling the so-called Racah--Speiser algorithm \cite{Rac64, Spe64}, also known as the Brauer--Klimyk algorithm \cite{Bra37, Kli68}, to compute the tensor product of two irreducible representations (here $\adj(\mathfrak{g})$ and $V_\mu$). This subsection follows the logic developed at the end of \cite[\S A.1]{LLF}.

The starting point is the Weyl character formula for the infinitesimal character of $V_\mu$ (still denoted by~$\chi_\mu$, but defined on $\lie{h}$ instead of $G$, by the formula $\chi_\mu(x) \coloneqq \Tr_{V_\mu}(x)$):
\begin{equation}\label{chimuexpr}
  \chi_\mu(x) = \frac{\sum_{\sigma\in W} (-1)^\sigma e^{(\sigma(\rho+\mu))\cdot x}}{\sum_{\sigma\in W} (-1)^\sigma e^{\sigma(\rho)\cdot x}}
\end{equation}
where $W$ is the Weyl group and $(-1)^\sigma$ denotes the signature of $\sigma$.
We use this expression to extend the definition of $\chi_\mu$ to arbitrary $\mu$ (although if $\mu$ is not a dominant weight, $\chi_\mu$ is no longer the character of any representation).

Multiplicativity of infinitesimal characters implies that the infinitesimal character of $\adj(\mathfrak{g})\otimes V_\mu$ is
\begin{equation}\label{char-adj-Vmu}
  \begin{aligned}
    \chi_{\adj(\mathfrak{g})\otimes V_\mu}(x) & = \chi_{\adj(\mathfrak{g})}(x) \chi_\mu(x) \\
    & = \frac{\sum_{\sigma\in W} (-1)^\sigma e^{(\sigma(\rho+\mu))\cdot x}\bigl( \rank(\mathfrak{g}) + \sum_{\anyroot\in\roots(\mathfrak{g})} e^{\anyroot\cdot x} \bigr)}{\sum_{\sigma\in W} (-1)^\sigma e^{\sigma(\rho)\cdot x}}
    \\
    & = \rank(\mathfrak{g}) \chi_\mu(x)
    + \sum_{\anyroot\in\roots(\mathfrak{g})} \frac{\sum_{\sigma\in W} (-1)^\sigma e^{(\sigma(\rho+\mu+\anyroot))\cdot x}}{\sum_{\sigma\in W} (-1)^\sigma e^{\sigma(\rho)\cdot x}} ,
  \end{aligned}
\end{equation}
where we have changed $e^{\anyroot\cdot x}$ to~$e^{\sigma(\anyroot)\cdot x}$ in the sum thanks to the Weyl group invariance of $\roots(\mathfrak{g})$.
This reduces to a sum of $\dim \mathfrak{g}$ (virtual) characters:
\begin{equation}\label{char-adj-Vmu-2}
  \chi_{\adj(\mathfrak{g})\otimes V_\mu}(x)
  = \rank(\mathfrak{g}) \chi_\mu(x) + \sum_{\anyroot\in\roots(\mathfrak{g})} \chi_{\mu+\anyroot}(x) .
\end{equation}

For $\mu$ deep enough in the dominant Weyl chamber, all shifted weights $\mu+\anyroot$ are dominant and this yields the decomposition into irreducible representations,
\begin{equation}\label{adjGammaVmu}
  \adj(\mathfrak{g})\otimes V_\mu = (V_\mu)^{\rank(\mathfrak{g})} \oplus \bigoplus_{\anyroot\in\roots(\mathfrak{g})} V_{\mu+\anyroot} , \quad
  \text{if all $\mu+\anyroot$ are dominant.}
\end{equation}
For $\mu$ near the boundary of the Weyl chamber, some $\mu+\anyroot$ exit the Weyl chamber and $V_{\mu+\anyroot}$ is meaningless.
To make sense of the term $\chi_{\mu+\anyroot}$ one should map the weight $\mu+\anyroot$ into the dominant Weyl chamber using a Weyl reflection.  An easy consequence of~\eqref{chimuexpr} is that characters flip signs under such reflections,
\begin{equation}
  \chi_{-\rho+\sigma(\rho+\nu)}(x) = (-1)^\sigma \chi_\nu(x) .
\end{equation}
In particular, if $\nu$ is a fixed point of such a reflection (namely is on a Weyl chamber boundary), then $\chi_\nu(x)$ vanishes.
This allows to write all summands in~\eqref{adjGammaVmu} in terms of characters of dominant weights and deduce the decomposition into irreducible representations.

\subsection{Generic situation and cancellations}
\label{sec:generic_and_cancellations}

More precisely, consider any root of $\lie{g}$, henceforth denoted by $\farrootSec$ instead of $\anyroot$ (this will help avoid confusion later, when we will apply Lemma~\ref{lem:prod_at_least_2}). Then one of the following five cases could in principle occur (compare the discussion following \cite[(A.5)]{LLF}, and see also \autoref{fig:root-g2}):
\begin{enumerate}
\item \label{itm:default} $\rho + \mu + \farrootSec \in \mathfrak{h}^{++}$, or equivalently $\mu + \farrootSec \in \mathfrak{h}^+$. In this case $\chi_{\mu + \farrootSec}$ is the character of an actual representation $V_{\mu + \farrootSec}$.
\item \label{itm:suicide} $\rho + \mu + \farrootSec$ lies on a wall of the dominant Weyl chamber, i.e.\ in $\mathfrak{h}^+ \setminus \mathfrak{h}^{++}$. It is then invariant by the reflection with respect to some root, which has signature $-1$, so $\chi_{\mu + \farrootSec} = -\chi_{\mu + \farrootSec} = 0$ (compare \cite[(A.6)]{LLF}).
\item \label{itm:kamikaze} $\rho + \mu + \farrootSec$ lies in some open Weyl chamber neighboring the dominant one, i.e.\ in $s_{\alpha_j} \mathfrak{h}^+$ for some simple root $\alpha_j$. In this case we have $\chi_{\mu + \farrootSec} = -\chi_{\mu + \orthrootSec}$, where $\orthrootSec := s_{\alpha_j}(\rho + \mu + \farrootSec) - \rho - \mu$. We shall see that in fact $\orthrootSec$ is always a root or $0$, so $\chi_{\mu + \farrootSec}$ cancels the simple summand $V_{\mu + \orthrootSec}$ (or decrements its multiplicity).
\item \label{itm:further_suicide} $\rho + \mu + \farrootSec$ lies on a wall of some nondominant Weyl chamber. This of course has the same outcome as case \ref{itm:suicide} (namely $\chi_{\mu + \farrootSec} = 0$), but is listed separately for clarity.
\item \label{itm:further_chamber} $\rho + \mu + \farrootSec$ lies in some further image of the open dominant Weyl chamber. We shall however see that in practice, this never happens.
\end{enumerate}
Now note that, by assumption, we already know that $\nu$, hence also $\nubar$, is dominant. So $\rho + \nubar$ itself always falls in case~\ref{itm:default} above; and we only need to rule out case~\ref{itm:further_chamber}, and to find which weights are cancelled by weights coming from case~\ref{itm:kamikaze}.

In other terms, it suffices to list all the weights of the form $\rho + \mu + \farrootSec$ for which we have
\begin{equation} \label{eq:negative_p}
-p := \langle \rho + \mu + \farrootSec, \alpha_j^\vee \rangle < 0
\end{equation}
for some simple root $\alpha_j$, which means that they are in cases \ref{itm:kamikaze}, \ref{itm:further_suicide} or possibly~\ref{itm:further_chamber}; and, for every such weight:
\begin{itemize}
\item We need to compute the cancelled weight
\begin{equation}
\label{eq:orthrootSec_def}
\mu + \orthrootSec = s_{\alpha_j}(\rho + \mu + \farrootSec) - \rho = \mu + \farrootSec + p \alpha_j,
\end{equation}
whose multiplicity is to be decremented.
\item To rule out case~\ref{itm:further_chamber}, we need to check that $\rho + \mu + \orthrootSec$ lies in $\lie{h}^+$. This entails verifying the inequality\begin{equation} \label{eq:single_step_to_dominant}
\langle \mu + \farrootSec + p \alpha_j, \alpha_k^\vee \rangle \geq -1
\end{equation}
for every simple root $\alpha_k \neq \alpha_j$ (recall that $\langle \rho, \alpha_k^\vee \rangle = 1$ by definition), the case $k = j$ being true by construction.
\item We do \emph{not}, however, need to explicity rule out case~\ref{itm:further_suicide} (which does in fact sometimes occur). Indeed in that case the character $\chi_{\mu+\gamma}=-\chi_{\mu+\beta}$ simply vanishes and does not affect any multiplicities.
\item We then accordingly decrement the multiplicity of $V_{\mu + \orthrootSec}$.
\end{itemize}

Recall that for all dominant $\mu$, for every root $\farrootSec$ and for every simple root $\alpha_j$, we have the identity and inequalities
\begin{equation}
\begin{cases}
\langle \rho, \alpha_j^\vee \rangle = 1; \\
\langle \mu, \alpha_j^\vee \rangle \geq 0; \\
\langle \farrootSec, \alpha_j^\vee \rangle \geq -3.
\end{cases}
\end{equation}
Hence \eqref{eq:negative_p} can hold, i.e.\ the sum of these three terms can be negative, only if $\langle \farrootSec, \alpha_j^\vee \rangle \leq -2$. This can happen only in a handful of cases, that we classify in the next subsection.

\subsection{Scalar products of roots and coroots}
\label{sec:prod_at_least_2}

\begin{lemma} \label{lem:prod_at_least_2}
For two roots $\farrootLem, \refrootLem$, the inequality
\begin{equation}
\langle\farrootLem,\refrootLem^\vee\rangle \leq -2
\end{equation}
happens in exactly the following situations:
\begin{enumerate}
\item \label{itm:equal_roots} whenever $\farrootLem = -\refrootLem$;
\item \label{itm:weakly_orth} whenever $\mathfrak{g}$ has a factor of type $B_n$, $C_n$ or $F_4$, $\refrootLem$ is a short root of that factor, and $\farrootLem = \orthrootLem - \refrootLem$ where $\orthrootLem$ is a root weakly orthogonal to $\refrootLem$ (see Table~\ref{tab:weak_orth_enum} for a complete list of such pairs $(\refrootLem, \orthrootLem)$ with simple $\refrootLem$);
\item \label{itm:g2_short_long} if $\mathfrak{g}$ has a factor of type $G_2$, $\refrootLem$ is a short root of that factor, and $\farrootLem$ is a long root (of that factor) immediately adjacent to $- \refrootLem$.
\end{enumerate}
In these three cases, the scalar product $\langle\farrootLem,\refrootLem^\vee\rangle$ is respectively equal to $-2$, $-2$ and $-3$.
\end{lemma}

\begin{proof}
Clearly we lose no generality in assuming that $\lie{g}$ is simple: indeed any two roots coming from different simple factors have scalar product 0.

If the two roots have the same length (e.g., for simply-laced~$\mathfrak{g}$), $\langle\farrootLem,\refrootLem^\vee\rangle=2\cos\theta \geq -2$, with equality if and only if $\farrootLem=-\refrootLem$.
If $\farrootLem$ is short and~$\refrootLem$ is long, then $|\langle\farrootLem,\refrootLem^\vee\rangle|\geq -2|\farrootLem|/|\refrootLem|>-2$.
Finally, if $\farrootLem$ is long and~$\refrootLem$ is short, then $|\farrootLem|/|\refrootLem|$ equals $\sqrt{2}$ for the Lie algebras $B_n$, $C_n$, $F_4$, and equals $\sqrt{3}$ for~$G_2$.

For the first case, we have $|\langle\farrootLem,\refrootLem^\vee\rangle|\geq -2\sqrt{2}>-3$ so $\langle\farrootLem,\refrootLem^\vee\rangle\geq -2$. Suppose that equality happens. Then $\orthrootLem := \farrootLem + \refrootLem$, being in the $\refrootLem$-root string of $\farrootLem$ (see \cite[Prop.~2.29]{Kna96}), must be a root or~$0$. The latter is impossible as $|\farrootLem| > |\refrootLem|$. We then compute $\langle \orthrootLem, \refrootLem \rangle = \langle \orthrootLem, \refrootLem^\vee \rangle = 0$, and note that $\orthrootLem + \refrootLem = \farrootLem + 2\refrootLem$ is also in the $\refrootLem$-root string of $\farrootLem$, hence a root. So $\orthrootLem$ is indeed a root weakly orthogonal to $\refrootLem$.

For the $G_2$~case, if we fix a short root, we may simply examine each of the six long roots (see \autoref{fig:root-g2}). We see that we then have $\langle\farrootLem,\refrootLem^\vee\rangle \in \{-3,0,3\}$, with the value $-3$ attained precisely for the two long roots adjacent to $-\refrootLem$.
\end{proof}

\subsection{Finding all possible cancellations}

We now finish the proof of Theorem~\ref{thm:adjoints_in_tensor_square}, by examining each of the three cases of Lemma~\ref{lem:prod_at_least_2}, finding the corresponding possible solutions of \eqref{eq:negative_p} and identifying the corresponding cancelled summands.

\begin{itemize}
\item First consider cases \ref{lem:prod_at_least_2}.\ref{itm:equal_roots} and \ref{lem:prod_at_least_2}.\ref{itm:weakly_orth}. Then $\langle \farrootSec, \alpha_j^\vee \rangle = -2$, so necessarily $\langle \mu, \alpha_j^\vee \rangle = 0$ and then $p = 1$.
\begin{itemize}
\item In case \ref{lem:prod_at_least_2}.\ref{itm:equal_roots} (i.e.\ $\farrootSec = -\alpha_j$), obviously 
$\orthrootSec = \farrootSec + \alpha_j = 0$. So we must decrease by~$1$ the coefficient of~$\chi_\mu$ in~\eqref{char-adj-Vmu-2}. This coefficient is initially equal to $\rank(\mathfrak{g})$, and is decremented exactly once for each root $\alpha_j$ such that $\langle \mu, \alpha_j^\vee \rangle = 0$. This yields the first line (case $\nu = \mubar$) of \eqref{eq:multiplicity_in_tensor_square}.
\item In case \ref{lem:prod_at_least_2}.\ref{itm:weakly_orth},
we compute that the reflected root $\orthrootSec = \farrootSec + p \alpha_j$ defined by \eqref{eq:orthrootSec_def} coincides with the $\orthrootLem = \farrootLem + \refrootLem$ of Lemma~\ref{lem:prod_at_least_2}.\ref{itm:weakly_orth} (applied to $\refrootLem = \alpha_j$).
To establish \eqref{eq:single_step_to_dominant}, note that both $\farrootSec$ and $s_{\alpha_j}(\farrootSec) = \farrootSec + 2\alpha_j$ are roots. By Lemma~\ref{lem:prod_at_least_2}, both of them have scalar product with $\alpha_k^\vee$ greater than or equal to $-2$ (the value $-3$ occurs only for $G_2$, which is not the case here). Moreover for $\farrootSec$ this inequality is strict: this can be checked on a case-by-case basis by looking at Table~\ref{tab:weak_orth_enum}. By averaging, we obtain \eqref{eq:single_step_to_dominant} for $\mu = 0$, hence also for any dominant $\mu$.

This case accounts precisely for case~\ref{itm:weakly_orth_overall} of Theorem~\ref{thm:adjoints_in_tensor_square}.
\end{itemize}

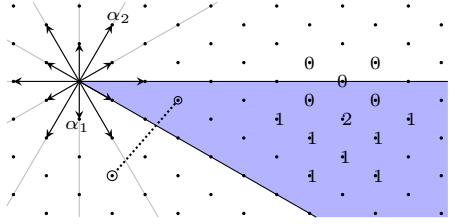
\begin{figure}\centering
  \begin{tikzpicture}[scale=.5,rotate=-90]
    \pgfmathsetmacro\rt{sqrt(3)}
    \fill[blue,opacity=.3] (0,0) -- (3.6,3.6*\rt) -- (3.6,5.6*\rt) -- (0,5.6*\rt) -- cycle;
    \draw (0,0) -- (3.6,3.6*\rt);
    \draw (0,0) -- (0,5.6*\rt);
    \draw[gray!50] (0,0) -- (-2.1,2.1*\rt);
    \draw[gray!50] (0,0) -- (-2.1,2.1/\rt);
    \draw[gray!50] (0,0) -- (-2.1,0);
    \draw[gray!50] (0,0) -- (-2.1,-2.1/\rt);
    \draw[gray!50] (0,0) -- (-1.1,-1.1*\rt);
    \draw[gray!50] (0,0) -- (0,-1.1*\rt);
    \draw[gray!50] (0,0) -- (1.1,-1.1*\rt);
    \draw[gray!50] (0,0) -- (1.1*3,-1.1*\rt);
    \draw[gray!50] (0,0) -- (3.6,0);
    \draw[gray!50] (0,0) -- (3.6,3.6/\rt);
    \coordinate(u) at (1,0);
    \coordinate(v) at (0,\rt);
    \foreach\x in {-2,...,3} \foreach\y in {-1,...,5} {
      \filldraw (\x,\y*\rt) circle (.03);
      \filldraw (\x+.5,\y*\rt+.5*\rt) circle (.03);
    }
    \foreach\th in {0,60,...,300} {
      \draw[-{stealth}, rotate=\th] (0,0) -- (1,0);
      \draw[-{stealth}, rotate=\th] (0,0) -- (0,\rt);
    }
    \node at (1.2,-.05) {\scriptsize$\alpha_1$};
    \node at (-1.7,.6*\rt) {\scriptsize$\alpha_2$};
    \node(A)[circle,draw,inner sep=1pt] at ($.5*(u)+1.5*(v)$) {};
    \node(B)[circle,draw,inner sep=1.3pt] at ($2.5*(u)+.5*(v)$) {};
    \draw[densely dotted,thick](A)--(B);
    \scriptsize
    \node at (1,3*\rt) {\;1};
    \node at (1.5,3.5*\rt) {\;1};
    \node at (2.5,3.5*\rt) {\;1};
    \node at (1,4*\rt) {\;\,2};
    \node at (2,4*\rt) {\;1};
    \node at (1.5,4.5*\rt) {\;1};
    \node at (2.5,4.5*\rt) {\;1};
    \node at (1,5*\rt) {\;1};
    \node at (.5,3.5*\rt) {0};
    \node at (-.5,3.5*\rt) {0};
    \node at (0,4*\rt) {0};
    \node at (.5,4.5*\rt) {0};
    \node at (-.5,4.5*\rt) {0};
  \end{tikzpicture}
  \caption{\label{fig:root-g2}Root system of~$G_2$ (arrows), its integral weights (lattice points), and the dominant Weyl chamber (shaded).
    The closest pair of weights (circle nodes) strictly inside this chamber, and strictly inside another chamber not adjacent to it, are at a distance~$\sqrt{7}|\alpha_1|$, longer than any root.
    On the right of the diagram are shown multiplicities ($2$, $1$ and~$0$) with which each representation appears in the tensor product of $\adj(\mathfrak{g})$ with the representation (whose highest weight shifted by~$\rho$ is) labeled~$2$: one character~$\chi_{\mu+\farrootLem}$ vanishes and four characters $\chi_{\mu+\farrootLem}$ cancel pairwise, as indicated by the explicit multiplicities~$0$ in the diagram.}
\end{figure}

\item Finally consider case \ref{lem:prod_at_least_2}.\ref{itm:g2_short_long} (see \autoref{fig:root-g2} for a picture). Then $\alpha_j = \alpha_1$ is the short simple root of some factor of type $G_2$, and $\langle \farrootSec, \alpha_1^\vee \rangle = -3$. This implies that $\langle \mu, \alpha_1^\vee \rangle$ may be either $0$ or $1$, which yields respectively $p = 2$ or $p = 1$.

We then have $\farrootSec\in\{\alpha_2,-3\alpha_1-\alpha_2\}$ (one of the long roots closest to $-\alpha_1$). For $\alpha_k \neq \alpha_j$ the scalar products are immediate to compute, and they obey~\eqref{eq:single_step_to_dominant}. Indeed for the long root $\alpha_2$ we have 
\begin{equation}
  \langle p\alpha_1+\alpha_2,\alpha_2^\vee\rangle = - p + 2 > -1 \ , \qquad
  \langle (p-3)\alpha_1-\alpha_2,\alpha_2^\vee\rangle = 3 - p - 2 \geq -1 \ ,
\end{equation}
and for the simple roots lying in other simple factors, the scalar product vanishes.

Computing $\orthrootSec = \farrootSec + p \alpha_1$, we find that this case accounts precisely for the cases \ref{itm:g2_1} and~\ref{itm:g2_2} of Theorem~\ref{thm:adjoints_in_tensor_square}.
\end{itemize}

\section{Proof of Theorem~\ref{thm:adj_in_sym_or_antisym}: dualities and symmetries}
\label{sec:dual_and_sym}

Assume now that $\mu = \nu$. Then Theorem~\ref{thm:adjoints_in_tensor_square} gives us the dimension of $\adj(\lie{g}) \otimes (V_\mu)^{\otimes 2}$; and we now want to count the dimension of its symmetric and antisymmetric subspaces.

We know that this dimension can be nonzero only if $\mu - \nubar = \mu - \mubar$ is a root or zero. In fact this difference can never be a root: indeed, it is invariant by the linear map $\mu \mapsto -\mubar$. This map is known to be the so-called longest word of the Weyl group, and is known to map all positive roots to negative roots and vice-versa. In particular it cannot leave any root invariant.

So we assume henceforth that $\mu = \mubar$, i.e.\ $V_\mu$ has an invariant (orthogonal or symplectic) quadratic form.

\subsection{Ad hoc root numbering}
\label{sec:ad_hoc_numbering}

For this section only, we deviate from the Bourbaki numbering. Instead, we renumber the simple roots $\alpha_i$, and correspondingly the fundamental weights $\varpi_i$ and the coordinates $k_i := \langle \mu, \alpha_i^\vee \rangle$ of the highest weight $\mu$, in an ad hoc way. First of all, we put all the vanishing coordinates at the end, so as to have
\begin{equation}
\mu = \sum_{i=1}^m k_{i} \varpi_{i}
\end{equation}
for some integer $m$ with $0 \leq m \leq \rank(\lie{g})$, with $k_{1}, \ldots, k_{m}$ all positive. By Theorem~\ref{thm:adjoints_in_tensor_square}, $m$ is thus the dimension of $(\adj(\lie{g}) \otimes (V_\mu)^{\otimes 2})^{\lie{g}}$.

We then sort the self-dual weights from the non-self-dual weights (whose nonzero coordinates always occur in pairs, as $\mu = \overline{\mu}$ by assumption), so as to have
\begin{equation}
\overline{i} =
\begin{cases}
{i + \ell} &\text{for } i = 1, \ldots, \ell; \\
{i - \ell} &\text{for } i = \ell+1, \ldots, 2\ell; \\
i &\text{for } i = 2\ell+1, \ldots, m,
\end{cases}
\end{equation}
for some integer $\ell$ (using the obvious convention $\varpi_{\overline{i}} = \overline{\varpi_i}$). In particular $k_{\overline{i}}=k_i$.

\subsection{Lie algebra action on factors}
\label{sec:preliminary_basis}

We seek to count trivial summands in the decomposition of $\adj(\mathfrak{g})\otimes S^2 V_\mu$ and $\adj(\mathfrak{g})\otimes\Lambda^2 V_\mu$ into irreducible representations of $\mathfrak{g}$.  By Schur's lemma, this counts the (linearly independent) representation morphisms (a.k.a.\ $\mathfrak{g}$-equivariant maps) from each of these two spaces to~$\CC$.  To determine this number, we shall exhibit a basis of morphisms $\psi_i\colon\adj(\mathfrak{g})\otimes V_\mu^{\otimes 2}\to\CC$ and analyse their symmetry properties.

We consider first a larger space $W = V_{k_{1} \varpi_{1}} \otimes \cdots \otimes V_{k_{m} \varpi_{m}}$ that contains $V_\mu$ as its main irreducible summand, and define the following family of morphisms, indexed by $i = 1, \ldots, m$:
\begin{equation}
\fundef{\phi_i \colon}
{\adj(\lie{g}) \otimes W}
{W}
{x\otimes v_1\otimes \dots \otimes v_m}
{v_1\otimes \dots\otimes x\cdot v_i\otimes \dots \otimes v_m,}
\end{equation}
i.e.\ $\phi_i$ makes $x$ (an element of the representation space of $\adj(\lie{g})$, which is just $\lie{g}$ itself) act only on the $i$-th component in the tensor product.

\subsection{Identification with dual}
\label{sec:duality_maps}

We would now like to ``lower the third index'', i.e.\ to transform these maps $\phi_i: \adj(\lie{g}) \otimes W \to W$ into some maps $\phi_i^\flat: \adj(\lie{g}) \otimes W^{\otimes 2} \to \CC$.

In order to do this, for each $i$, let us choose some explicit representation isomorphism
\begin{equation}
s_{i}: V_{k_{i} \varpi_{i}} \to V_{k_{i} \overline{\varpi_{i}}}^*
\end{equation}
(by Schur's lemma, the choice resides only in the normalization). They satisfy
\begin{align}
\label{eq:g-equivariance}
  s_{i}(x \cdot w) v
  &= (x \cdot s_{i}(w)) v \nonumber \\
  &= - s_{i}(w) (x \cdot v)
\end{align}
for each $i = 1, \ldots, m$, $x \in \lie{g}$, $v \in V_{k_{i} \overline{\varpi_{i}}}$ and $w \in V_{k_{i} \varpi_{i}}$. We then define
\begin{equation}
\label{eq:phi_flat_definition}
\fundef{\phi_i^\flat \colon}
{\adj(\lie{g}) \otimes W^{\otimes 2}} 
{\CC}
{\displaystyle x \otimes \left(\bigotimes_{i=1}^m v_{i}\right) \otimes \left(\bigotimes_{i=1}^m w_{i}\right)}
{\displaystyle s_{\overline{i}}(w_{\overline{i}}) (x \cdot v_{i}) \prod_{j \neq i} s_{\overline{{j}}}(w_{\overline{{j}}}) v_{{j}}.}
\end{equation}

\subsection{Symmetry properties}
\label{sec:symmetry_properties}

We now seek the symmetry properties of $\phi^\flat_i$. Note that for every $i$, the dual map $s_{i}^*: V_{k_{i} \overline{\varpi_{i}}} \to V_{k_{i} \varpi_{i}}^*$ is also a representation morphism, so by Schur's lemma satisfies
\begin{equation}
\label{eq:dual_isomorphism}
s_{i}^* = \eps_{i} s_{\overline{i}} 
\end{equation}
for some family of scalars $\eps_{i}$, which moreover obeys $\eps_{i}\eps_{\overline{i}} = 1$. For each of the self-dual indices (namely $2\ell<i\leq m$), this $\eps_{i}$ does not depend on the normalization of $s_{i} = s_{\overline{i}}$, and is in fact the Frobenius-Schur indicator $\FrobSch(k_{i} \varpi_{i}) = \FrobSch(\varpi_{i})^{k_{i}}$ of the representation $V_{k_{i} \varpi_{i}}$: we remind that it can be defined as
\begin{equation}
\label{eq:Frob_Schur_def}
\FrobSch(\mu) := \begin{cases}
+1 &\text{if $V_\mu$ is orthogonal;} \\
-1 &\text{if $V_\mu$ is symplectic;} \\
0 &\text{if $V_\mu$ fails to be self-dual (irrelevant here).} \\
\end{cases}
\end{equation}
For the indices that come in dual pairs (namely $i \leq 2\ell$), without loss of generality, one can adjust one of the maps $s_{i}$ of each pair (say the one with $i > \ell$) by a scalar factor so as to have $\eps_i = 1$ for all $i = 1, \ldots, \ell$; this will then also yield $\eps_i = \eps_{i-\ell}^{-1} = 1$ for all $i = \ell+1, \ldots, 2\ell$. To summarize, we get:
\begin{equation}
\label{eq:eps_i_signs}
\eps_{i} = \begin{cases}
+1 &\text{for } i = 1, \ldots, 2\ell; \\
\FrobSch(k_{i} \varpi_{i}) &\text{for } i = 2\ell+1, \ldots, m.
\end{cases}
\end{equation}
From the multiplicativity of the Frobenius-Schur indicator and from the well-known fact that $\FrobSch(\varpi + \overline{\varpi}) = +1$ for every integral weight $\varpi$, we then deduce the identity
\begin{equation}
\label{eq:Frob_Schur_identity}
\prod_{i=1}^m \eps_{i} = \FrobSch(\mu).
\end{equation}
Now take any pure tensor $x \otimes v \otimes w \in \adj(\lie{g}) \otimes W^{\otimes 2}$, with $v = \bigotimes_{i=1}^m v_{i}$ and $w = \bigotimes_{i=1}^m w_{i}$. We compute, for each $i = 1, \ldots, m$:
\begin{align}
\phi^\flat_i \left( x \otimes w \otimes v \right)
&= s_{\overline{i}}(v_{\overline{i}}) (x \cdot w_{i}) \prod_{j \neq i} s_{\overline{{j}}}(v_{\overline{{j}}}) w_{{j}} &\text{by def. \eqref{eq:phi_flat_definition}} \nonumber \\
&= s_{\overline{i}}^*(x \cdot w_{i}) (v_{\overline{i}}) \prod_{j \neq i} s_{\overline{{j}}}^*(w_{{j}}) v_{\overline{{j}}} \nonumber \\
&= \left( \prod_{i=1}^m \eps_{i} \right) s_{i}(x \cdot w_{i}) (v_{\overline{i}}) \prod_{j \neq i} s_{{j}}(w_{{j}}) v_{\overline{{j}}}  &\text{by \eqref{eq:dual_isomorphism}} \nonumber \\
&= \FrobSch(\mu)\; s_{i}(x \cdot w_{i}) (v_{\overline{i}}) \prod_{j \neq \overline{i}} s_{\overline{{j}}}(w_{\overline{{j}}}) v_{{j}}  &\text{by \eqref{eq:Frob_Schur_identity}} \nonumber \\
&= - \FrobSch(\mu)\; s_{i}(w_{i}) (x \cdot v_{\overline{i}}) \prod_{j \neq \overline{i}} s_{\overline{{j}}}(w_{\overline{{j}}}) v_{{j}}  &\text{by $\lie{g}$-equivar. \eqref{eq:g-equivariance}} \nonumber \\
&= - \FrobSch(\mu) \phi^\flat_{\overline{i}} \left( x \otimes v \otimes w \right) &\text{by def. \eqref{eq:phi_flat_definition}}.
\end{align}
Thus we obtain that:
\begin{itemize}
\item for every self-dual $i$ (i.e.\ $2\ell < i \leq m$), the map $\phi^\flat_i$ has symmetry properties \emph{opposite} to those of the invariant quadratic form on $V_\mu$: i.e.\ it is symmetric if $V_\mu$ is symplectic, and antisymmetric if $V_\mu$ is orthogonal;
\item for every dual pair $(i, i+\ell)$ (with $i = 1, \ldots, \ell$), among the two maps $\phi^\flat_i \pm \phi^\flat_{i+\ell}$, exactly one is symmetric and exactly one is antisymmetric.
\end{itemize}

\subsection{Suitable basis of $\adj(\lie{g}) \otimes (V_\mu)^{\otimes 2}$}
\label{sec:actual_basis}

For each $i = 1, \ldots, m$, let $\psi_i$ be the restriction of $\phi^\flat_i$ to $\adj(\lie{g}) \otimes V_\mu^{\otimes 2}$, where we identify $V_\mu = V_{\sum_{i=1}^m k_{i} \varpi_{i}}$ with the main irreducible summand of $W = \bigotimes_{i=1}^m V_{k_{i} \varpi_{i}}$. It still remains to verify that these maps are linearly independent.

For each $i = 1, \ldots, m$, let $v_i$ (resp.\ $w_i$) be some highest (resp.\ lowest) weight vector of $V_{k_{i} \varpi_{i}}$. Then it is easy to see that, for each $i$, $s_{\overline{i}}(w_{\overline{i}})$ is still a lowest weight vector of $V^*_{k_{i} \varpi_{i}}$, and to deduce that
\[\forall i = 1, \ldots, m,\quad s_{\overline{i}}(w_{\overline{i}}) v_{i} \neq 0.\]
Furthermore, we also see that $v := v_1 \otimes \cdots \otimes v_m$ (resp.\ $w := w_1 \otimes \cdots \otimes w_m$) is a highest (resp.\ lowest) weight vector of the tensor product representation, hence both of them lie in the main irreducible summand $V_\mu$. Now take any element $x \in \lie{h} \subset \lie{g}$; we then compute, for every $i = 1, \ldots, m$:
\begin{equation}
\psi_i(x \otimes v \otimes w) = \phi^\flat_i(x \otimes v \otimes w) = k_{i} \varpi_{i}(x) C,
\end{equation}
where $C := \prod_{i=1}^m s_{\overline{i}}(w_{\overline{i}}) v_{i} \in \CC$ is some nonzero constant. Since the linear forms $k_{i} \varpi_{i}$ are linearly independent, so are the $\psi_i$.

It follows that the $m$ maps $\psi_1 + \psi_{\ell+1}, \psi_1 - \psi_{\ell+1}, \ldots, \psi_\ell + \psi_{2\ell}, \psi_\ell - \psi_{2\ell},\allowbreak \psi_{2\ell+1}, \ldots, \psi_m$ are linearly independent as well. But the previous discussion shows that, among these maps, the number of symmetric and antisymmetric ones precisely agrees with the respective right-hand-sides of the two lines of \eqref{eq:adj_in_sym_or_antisym}. This concludes the proof of \autoref{thm:adj_in_sym_or_antisym}.

\subsection{Reproducing results of King and Wybourne}
\label{ssec:KW}

We now specialize our \autoref{thm:adj_in_sym_or_antisym} on dimensions of $G$-invariant subspaces of $\adj(G)\otimes S^2 V$ and $\adj(G)\otimes\Lambda^2 V$ to each classical and exceptional Lie group to efficiently reproduce the results in~\cite{KW}, prove one of their conjectures and disprove the other.

For brevity we denote by $b(\mu)$, $b_S(\mu)$, and $b_\Lambda(\mu)$ the dimensions of $G$-invariant subspaces of $\adj(G)\otimes V^{\otimes 2}$, $\adj(G)\otimes S^2 V$ and $\adj(G)\otimes\Lambda^2 V$, respectively.
Throughout, $V$~is an irreducible representation of highest weight~$\mu$, and is either orthogonal or symplectic, as otherwise these dimensions simply vanish.
We refer to \cite[Exercise~4.3.13]{OV90} for a table of the Frobenius-Schur indicators.
We return to the standard Bourbaki numbering of fundamental weights~$\varpi_i$.
\begin{enumerate}
\item For the group $G=SU(k+1)$, conjugation is $\overline{\varpi_i}=\varpi_{k+1-i}$.
  The dimension $b(\mu)$ given in~\eqref{eq:multiplicity_in_tensor_square} is odd if and only if $k$ is odd and $\langle\mu,\alpha_{(k+1)/2}\rangle>0$ since that is the only self-dual fundamental weight.
  In that case, \eqref{eq:adj_in_sym_or_antisym} states that $\{b_S(\mu),b_\Lambda(\mu)\}=\{(b(\mu)\pm 1)/2\}$, with $b_S(\mu)$ being the larger (resp.\ smaller) one if $V$ is symplectic (resp.\ orthogonal).
  Otherwise, $b_S(\mu)=b_\Lambda(\mu)$ is immediate from~\eqref{eq:adj_in_sym_or_antisym}.
  This coincides with \cite[Prop.~5.4]{KW}.

\item The group $G=Spin(4k+2)$\footnote{Note that King and Wybourne talk everywhere about groups $SO(n)$ (which, strictly speaking, only have ``half'' of the possible representations) but work with representations of the corresponding Lie algebras.  Following our convention that $G$ is simply connected, here we rephrase everything in terms of $Spin(n)$.} has a pair of fundamental weights $\varpi_{2k},\varpi_{2k+1}$ that are exchanged by conjugation, while the other fundamental representations are self-dual and orthogonal.  The group thus does not have any symplectic representation, so that \eqref{eq:adj_in_sym_or_antisym} simplifies to $b_S(\mu) = \delta_{\langle\mu,\alpha_{2k}^{\vee}\rangle>0}$ and $b_\Lambda(\mu)=b(\mu)-b_S(\mu)$, which reproduces \cite[Prop.~6.7]{KW}.

\item For the groups $G=Spin(2k+1),USp(2k),Spin(4k),E_7,E_8,F_4,G_2$, all fundamental weights are self-dual so $\{b_S(\mu),b_\Lambda(\mu)\}=\{0,b(\mu)\}$, with $b_S(\mu)$ (resp.\ $b_\Lambda(\mu)$) vanishing for orthogonal (resp.\ symplectic) representations (this statement can also be found as \cite[Cor.~3.4]{PY08}).  Furthermore, the spin groups $Spin(n)$ with $n \not\equiv 3,4,5 \pmod{8}$ and the exceptional groups $G_2,F_4,E_8$ only have orthogonal representations so one always has $b_S(\mu)=0$. This is as stated in \cite[Prop.~6.3--6.5]{KW} for the classical groups, in \cite[Prop.~7.1--7.2]{KW} for $G_2$ and~$E_7$, and conjectured for $F_4$ and $E_8$ in \cite[Conj.~7.3]{KW}.

\item Finally, for the $E_6$ group our general theorem disproves \cite[Conj.~7.4]{KW}, which proposed ``on the basis it has to be said of very little data'' that $b_S(\mu)$ and $b_\Lambda(\mu)$ were respectively the floor and ceiling of $b(\mu)/2$.  The $E_6$ group has no symplectic representation and one has
\begin{equation*}
  b_S(\mu) = \delta_{\langle\mu,\alpha_1^\vee\rangle>0} + \delta_{\langle\mu,\alpha_3^\vee\rangle>0} , \quad
  b_\Lambda(\mu) = b_S(\mu) + \delta_{\langle\mu,\alpha_2^\vee\rangle>0} + \delta_{\langle\mu,\alpha_4^\vee\rangle>0} ,
\end{equation*}
where $b_S(\mu)$ receives one contribution for each of the two pairs of dual coroots, and $b_\Lambda(\mu)$ additionally receives one contribution for each of the self-dual coroots.
If $b$ is odd, then $b_\Lambda(\mu)=b_S(\mu)+1$ as proposed in \cite{KW}, while if $b$ is even one can either have $b_\Lambda(\mu)=b_S(\mu)$ or $b_\Lambda(\mu)=b_S(\mu)+2$.
\end{enumerate}

We may also answer the following question, raised in the introduction of \cite{KW}: ``When does the adjoint $\theta$ occur only in the symmetric or only in the antisymmetric part of the Kronecker square of a given irreducible representation $\lambda$ of $G$?'' From \eqref{eq:adj_in_sym_or_antisym}, it immediately follows that this happens if and only if its highest weight  lies in the subspace generated by the self-dual fundamental weights. In particular, it is true for \emph{all} representations when the opposition involution is trivial (see the case 3 above).

\bibliographystyle{alpha}
\bibliography{/home/ilia/Documents/Travaux_mathematiques/mybibliography.bib}

\noindent I. Smilga, Mathematical Institute, University of Oxford, United Kingdom 

\noindent E-mail: \url{ilia.smilga@normalesup.org}

\noindent B. Le Floch, CNRS and LPTHE, Sorbonne Universite, Paris, France

\noindent E-mail: \url{blefloch@lpthe.jussieu.fr} 

\end{document}